\documentclass[11pt, reqno]{article}
\pdfoutput=1

\usepackage{amsthm, amssymb, amsmath}
\usepackage[colorlinks=true]{hyperref}
 \usepackage{fullpage}
\usepackage{verbatim}
\usepackage{graphicx}
\usepackage{epsfig}
\usepackage{color}
\usepackage[all]{xy}

\newtheorem{thm}{Theorem}[section]
\newtheorem{lem}[thm]{Lemma}
\newtheorem{cor}[thm]{Corollary}
\newtheorem{prop}[thm]{Proposition}

\newtheorem*{conjecture*}{Conjecture}

\theoremstyle{remark} 
\newtheorem*{question*}{Question}

\newtheorem{remark}[thm]{Remark}
\newtheorem{example}[thm]{Example}

\theoremstyle{definition}

\numberwithin{equation}{section}  

\newcommand{\OO}{\mathcal{O}}    
\newcommand{\FF}{\mathbb{F}}      
\newcommand{\ZZ}{\mathbb{Z}}     
\newcommand{\PP}{\mathbb{P}}      
\newcommand{\Aff}{\mathbb{A}}      
\newcommand{\QQ}{\mathbb{Q}}      
\newcommand{\CC}{\mathbb{C}}      
\newcommand{\mm}{\mathfrak{m}}   

\newcommand{\be}{\begin{equation}}
\newcommand{\ee}{\end{equation}}
\newcommand{\benn}{\begin{equation*}}
\newcommand{\eenn}{\end{equation*}}
\newcommand{\ba}{\begin{aligned}}
\newcommand{\ea}{\end{aligned}}
\newcommand{\bbm}{\begin{bmatrix}}
\newcommand{\ebm}{\end{bmatrix}}
\newcommand{\bpm}{\begin{pmatrix}}
\newcommand{\epm}{\end{pmatrix}}
\newcommand{\bi}{\begin{itemize}}
\newcommand{\ei}{\end{itemize}}
 
\newcommand{\xhelp}[1]{\textbf{Xander: #1}}

\newcommand{\Spec}{\operatorname{Spec}}

\newcommand{\an}[1]{\operatorname{an}}  
 \newcommand{\Rat}{\operatorname{Rat}}    

\newcommand{\PGL}{\mathrm{PGL}}

\newcommand{\Crit}[1]{\mathrm{Crit}\left(#1\right)_{\mathrm{f}}}
\newtheorem*{obs*}{Key Fact}
\newcommand{\Insep}{\mathrm{Insep}}
\newcommand{\Res}{\mathrm{Res}}

\newcommand{\uni}{U_d^{(\infty)}}
\newcommand{\sign}{\mathrm{sign}}
\newcommand{\Poly}{\mathrm{Poly}}
\newcommand{\SL}{\mathrm{SL}}
\newcommand{\UU}{\mathfrak{U}}

\title{Rational Functions with a Unique Critical Point}
\author{Xander Faber \\
Department of Mathematics \\
University of Georgia \\
Athens, GA  30602 \\ 
\texttt{xander@math.uga.edu}
}

 \date{}

\begin{document}
\maketitle
	\begin{abstract}
		Over an algebraically closed field of positive characteristic, there exist rational functions with only one 
		  critical point. We give an elementary characterization of these
		 functions in terms of their continued fraction expansions. 
		Then we use this tool to discern some of the basic geometry of the space of unicritical rational functions, as
		well as its quotients by the $\SL_2$-actions of conjugation and postcomposition. 
		We also give an application to dynamical systems with 
		restricted ramification defined over non-Archimedean fields of positive residue characteristic. \\
		
        (secondary) \newline
	\end{abstract}

\maketitle


\section{Introduction}

    A nonconstant rational function $\phi(z) = f(z) / g(z)$  defined over the complex numbers is either an automorphism of $\PP^1(\CC) = \CC \cup \{\infty\}$, or else it must have at least two critical points --- i.e., points of $\PP^1(\CC)$ at which the associated map on tangent spaces vanishes. 
But if $F$ is an algebraically closed field of positive characteristic, then there exist rational functions with only one critical point, where the map must necessarily be wildly ramified. Let us call them \textbf{unicritical functions}. For example, the function $\phi(z) = z^p - z$ has $\infty$ as its only critical point. 

	In this article we study some of the geometric and dynamical properties of unicritical rational functions. This can be viewed as a characteristic~$p$ analogue of the work of Milnor \cite{Milnor_Two_Critical_Points_2000} and others on rational functions over $\CC$ with exactly two critical points. Our first goal is to characterize those functions with a unique critical point in terms of their continued fraction expansions. 
Given a rational function $\phi \in F(z)$, repeated use of the division algorithm produces unique polynomials $f_0, \ldots, f_n$ with coefficients in $F$ such that 
	\[
		\phi(z) =  f_0(z) + \frac{1}{f_1(z) + \frac{1}{\ddots + \frac{1}{f_{n}(z)} } }.
	\]
The expression on the right is called the \textbf{continued fraction expansion} of $\phi$; for brevity it is often written as $\phi = [f_0, f_1, \ldots, f_n]$. The construction of the continued fraction expansion forces $f_1, \ldots, f_n$ to be nonconstant (but not $f_0$).

\newpage
    
\begin{thm}
\label{Thm: No Finite Critical}
    Suppose $F$ is an algebraically closed field of positive characteristic~$p$. A rational function 
    $\phi \in F(z)$ has no finite critical point if and only if its continued fraction expansion has the form 
	\be
	\label{Eq: Special Form} 
    	\phi(z) = [q_0(z^p), q_1(z^p), \ldots, q_n(z^p)+ az]
	\ee 
	for some integer $n \geq 0$, polynomials 
    $q_0, q_1, \ldots, q_n \in F[z]$, and a nonzero element $a \in F$. 
\end{thm}

	For $\phi$ as in~\eqref{Eq: Special Form}, the $j$th \textit{convergent} of $\phi$ is given by $[q_0(z^p), \ldots, q_j(z^p)]$ for $j < n$. Note that these are inseparable rational functions. By analogy with the classical theory of continued fractions, one expects the convergents of a rational function to be good approximations to it. Hence we may think of $\phi$ as being well-approximated by inseparable functions. 

\begin{cor}
	Let $F$ be an algebraically closed field of positive characteristic $p$, and let $\phi \in F(z)$ be a rational function with a single critical point $c$. Then there exist polynomials $q_0, \ldots, q_n \in F[z]$ and a nonzero element $a \in F$ such that  $\phi \circ \sigma (z) = [q_0(z^p), q_1(z^p), \ldots, q_n(z^p)+ az]$, where $\sigma(z) = z$ if $c = \infty$ and $\sigma(z) = c + 1 / z $ otherwise.   
\end{cor}

\begin{proof}
	The function $\phi \circ \sigma$ has no finite critical point, so the theorem applies. 
\end{proof}



	A simple induction shows that the degree of the rational function $\phi = [f_0, f_1, \ldots, f_n] \in F(z)$ is
	\[
		\deg(\phi) = \max\{\deg(f_0), 0\} + \sum_{1 \leq i \leq n} \deg(f_i).
	\]
The above theorem immediately places an interesting restriction on the degree of a rational function with a single critical point.

\begin{cor}
\label{Cor: Congruence}
		Let $F$ be an algebraically closed field of positive characteristic $p$, and let $\phi \in F(z)$ be a unicritical rational function. Then 
		\[
			\deg(\phi) \equiv 0 \text{ or } 1 \pmod p.
		\]  
\end{cor}

	The space of rational functions of degree~$d$, denoted $\Rat_d$, is realized as a Zariski open subset of  $\PP^{2d+1}$ via the identification
	\[
		\phi(z) = \frac{a_dz^d + \cdots + a_0}{b_dz^d + \cdots + b_0} \mapsto 
		(a_d : \cdots : a_0 : b_d:  \cdots : b_0).
	\]
Over the complex numbers, the topology of $\Rat_d$ was first investigated by Segal \cite{Segal_Rational_Functions_1979} and then later by Milnor in the case $d = 2$ \cite{Milnor_Rational_Maps_1993}. Silverman has shown that it is a fine solution to the relevant moduli problem in the context of algebraic geometry \cite{Silverman_Rational_Maps_1998}. Here we will define a subvariety $U_d \subset \Rat_d$ whose geometric points correspond to rational functions of degree~$d$ possessing a unique critical point; we call it the \textbf{unicritical locus}. The coefficients of the polynomials in the continued fraction expansion may be used as generic coordinates on the unicritical locus; we exploit this observation to deduce some basic features of the geometry of $U_d$.

\begin{thm}
\label{Thm: Unicritical Geometry}
	Fix a prime field $\FF_p$ and an integer $d > 1$. The unicritical locus $U_d$ is empty if $d \not\equiv 0,1 \pmod p$, and otherwise it is an irreducible rational quasi-projective variety over $\FF_p$ satisfying
	\[
		\dim(U_d) = \begin{cases} 
				3 + 2d / p & \text{if } p \mid d \\
				4 + 2(d-1) / p & \text{if } p \mid d-1.
			\end{cases}
	\]
\end{thm}

	Over an arbitrary field, one can use the continued fraction expansion of a rational function to stratify the space $\Rat_d$. More precisely, fix a tuple of nonnegative integers $\kappa = (\kappa_0, \ldots, \kappa_n)$, and define $\Rat_d(\kappa)$ to be the space of rational functions $\phi$ with continued fraction expansion of the form $[f_0(z), \ldots, f_n(z)]$, where $\kappa_0 = \max\{ \deg(f_0), 0\}$ and $\kappa_i = \deg(f_i)$ for $i > 0$. Then $\Rat_d$ is the union of finitely many strata $\Rat_d(\kappa)$. This observation will be useful in \S\ref{Sec: Unicritical Geometry} for proving Theorem~\ref{Thm: Unicritical Geometry}. 
	
	Unfortunately, the continued fraction expansion does not behave well under composition of rational functions, and so this stratification does not descend to the space of dynamical systems $\mathrm{M}_d = \Rat_d / \SL_2$, which is constructed as a quotient by the action of $\SL_2$ on $\Rat_d$ by conjugation. See \cite{Milnor_Rational_Maps_1993, Silverman_Rational_Maps_1998, Levy_M_d_rational_2011} for more details and references on the topology and geometry of the space $\mathrm{M}_d$. Nevertheless, the unicritical locus $U_d$ is conjugation invariant, and so the space of unicritical conjugacy classes $\UU_d = U_d / \SL_2  \subset \mathrm{M}_d$ is well defined. Evidently $\UU_d$ is unirational, but the techniques used to prove Theorem~\ref{Thm: Unicritical Geometry} imply more.
	
\begin{thm}
\label{Thm: Unicritical Conjugacy}
	Fix a prime field $\FF_p$ and an integer $d > 1$ satisfying $d \equiv 0$ or $1 \pmod{p}$. 
	The space of unicritical conjugacy classes $\UU_d$ is an irreducible rational variety over $\FF_p$.
\end{thm}

	Consider the following classical question in algebraic geometry: 
	
\begin{question*} Given distinct points $P_1, \ldots, P_n \in \PP^1$ and integers $d, e_1, \ldots, e_n > 1$ satisfying $e_i \leq d$ and $\sum (e_i - 1) \leq 2d -2$, how many classes of separable degree-$d$ rational functions are there that ramify to order exactly $e_i$ at $P_i$ and are unramified elsewhere, modulo postcomposition by automorphisms of $\PP^1$? 
\end{question*}

Each ramification condition corresponds to a special Schubert cycle in the Grassman variety of $2$-planes in $(d+1)$-space, and so the question may be rephrased as asking for the size of their intersection. Over the complex numbers, the formulas of Hurwitz and Pieri imply that such rational functions exist if and only if $\sum (e_i - 1) = 2d-2$.
Eisenbud and Harris proved that these Schubert cycles intersect properly, and hence the number of postcomposition classes is finite \cite{Eisenbud-Harris_Divisors_General_Curves_1983}. Assuming the points $P_i$ are in general position, explicit formulas for the number of classes were given by Goldberg for simple ramification \cite{Goldberg_Catalan_1991}  and Scherbak in general~\cite{Scherbak_Critical_Points_2002}. 

	In characteristic~$p$, the question is more subtle, even after taking into account wild ramification. For example, the set of postcomposition classes need not be finite, in which case one should instead ask for its dimension. Osserman used degeneration techniques to give a complete answer when $p > d$ or $p < e_i$ for all $i$ \cite{Osserman_Ramification_char_p_2006} and when the ramification is not too wild \cite[Thm.~1.3]{Osserman_Deformations_2005}. We are able to contribute a few new cases:
	
\begin{thm}
\label{Thm: Classes}
	Fix $d, e > 1$ such that $d \equiv 0 \text{ or } 1 \pmod{p}$ and $p \mid e \leq d$, and fix a point $c \in \PP^1(F)$. Let $U_{d,e}^{(c)}$ be the space of  rational functions of degree~$d$ that are ramified at $c$ with index $e$ and unramified elsewhere. Then
	\[
		\dim\left(\SL_2 \backslash U_{d,e}^{(c)} \right) = \begin{cases} 
				(2d - e) / p & \text{if } p \mid d \\
				1 + (2d - 2 - e) / p & \text{if } p \mid (d-1).
			\end{cases}
	\]
In particular, $U_{d,e}^{(c)}$ is nonempty under these hypotheses. 
\end{thm}

	The original motivation for this note came from some questions that arise in the theory of ramification loci for rational functions on the Berkovich projective line. Rather than delve into this technical subject here, we instead present an application of Theorem~\ref{Thm: No Finite Critical} that may be stated in a more classical context. (The interested reader may also look at the discussion of the locus of total ramification in \cite{Faber_Berk_RamI_2011}.) 
	
	Let $k$ be an algebraically closed field that is complete with respect to a nontrivial non-Archime\-dean absolute value $|\cdot|$. We assume further that $k$ has residue characteristic $p > 0$. For example, $k$ could be the completion of an algebraic closure of the field of $p$-adic numbers $\QQ_p$ (denoted $\CC_p$) or of the field of formal Laurent series $\FF_p(\!(t)\!)$. Write $\OO_k = \{x \in k : |x| \leq 1\}$ and $\mm = \{x \in \OO_k : |x| < 1\}$ for the valuation ring of $k$ and its maximal ideal, respectively, and let $\tilde k = \OO_k / \mm$ be its residue field. Writing $\PP^1(k) = k \cup \{\infty\}$ (and similarly for $\PP^1(\tilde k)$), there is a canonical reduction map $\mathrm{red}: \PP^1(k) \to \PP^1(\tilde k)$ given by $\mathrm{red}(x) =x \pmod{\mm}$ if $x \in \OO_k$, and $\mathrm{red}(x) = \infty$ otherwise. 
	
\begin{cor}
\label{Cor: Application}
	Let $\varphi \in k(z)$ be a nonconstant rational function with good reduction such that all of the critical points of $\varphi$ have the same image in $\PP^1(\tilde k)$.  Then 
	\[
		\deg(\varphi) \equiv 0 \text{ or } 1 \pmod p.
	\] 
\end{cor}

See \S\ref{Sec: Application} for the definition of good reduction and for the proof of the corollary.  In \S\ref{Sec: Background}, we prove Theorem~\ref{Thm: No Finite Critical}. We endow $U_d$ with the structure of a quasi-projective variety in \S\ref{Sec: U_d definition}. In \S\ref{Sec: Preliminaries} we set up the framework for proving Theorem~\ref{Thm: Unicritical Geometry},~\ref{Thm: Unicritical Conjugacy}, and~\ref{Thm: Classes}, which is then accomplished in \S\ref{Sec: Proofs}. We illustrate these techniques with the special case of unicritical functions of degree~$p$ in \S\ref{Sec: Example}.   \\

\noindent \textbf{Acknowledgments.} The author was supported by a National Science Foundation Postdoctoral Research Fellowship during this work. He would like to thank Joseph Silverman and Alon Levy for the discussions that led to the statement and proof of Theorem~\ref{Thm: Unicritical Conjugacy}.


\section{Background and Proof of the Characterization}
\label{Sec: Background}

    Let $F$ be an algebraically closed field, and let $\phi \in F(z)$ be a rational function. For $x \in \PP^1(F)$, choose a fractional linear transformation  $\sigma$ with coefficients in $F$ such that $\sigma(\phi(x)) \neq \infty$. If $x \neq \infty$, we say that $x$ is a \textbf{finite critical point} of $\phi$ if $\frac{d(\sigma \circ \phi)}{dz} (x) = 0$. If $x = \infty$, we say that it is an \textbf{infinite critical point} if $\frac{d(\sigma \circ \phi(1/z))}{dz}\Big|_{z=0} = 0$. This definition does not depend on the choice of $\sigma$.
(Note that $x$ is a critical point of $\phi$ if and only if the induced linear map $\phi_*$ on tangent spaces vanishes at $x$.)  The set of \textbf{finite critical points} is denoted $\Crit{\phi}$.

    If $F$ is a field of positive characteristic $p$, a rational function  $\phi$ with coefficients in $F$ is called \textbf{inseparable} if $\phi(z) = \psi(z^p)$ for some rational function $\psi \in F[z]$. When $\phi$ is nonconstant, inseparability is equivalent to saying that the extension of function fields $F(z) / F(\phi(z) )$ is not separable in the sense of field theory.

\begin{lem}
\label{Lem: Kernel}
    Let $F$ be a field of positive characteristic $p$, and let $D: F(z) \to F(z)$ be the formal derivative map. Then the kernel of $D$ is precisely the subfield of inseparable rational functions $F(z^p)$.
\end{lem}

\begin{remark}
	The lemma is an immediate consequence of the Hurwitz formula; however, for the reader's convenience we give an elementary proof. 
\end{remark}

\begin{proof}
    We begin by treating the polynomial case. Let $f \in F[z]$, and write 
      \[
        f(z) = \sum a_j z^j = \sum_{j: \ p \mid j} a_j (z^p)^{j/p} + \sum_{j: \ p \nmid j} a_j z^j = f_1(z^p) + f_2(z).
      \]
Then $D(f_1(z^p) ) = 0$, so that $D(f) = D(f_2)$. Since $p$ does not divide the exponent of any monomial of $f_2$, we see that $D(f) = 0$ if and only if $f_2 = 0$. 

	Returning to the general case, it is evident that the field of inseparable rational functions lies in the kernel of $D$. So let us suppose that $\phi \in \ker(D)$ is arbitrary, and write $\phi = f / g$ with $f$ and $g$ sharing no common linear factor. We may assume that $f$ is nonzero. The quotient rule shows that
	\begin{equation}
	\label{Eq: Anti-Leibniz}
		D(f)g = fD(g).
	\end{equation}
If $D(g) = 0$, then $D(f)$ must also vanish, and the previous paragraph allows us to write $f(z) = f_1(z^p)$ and $g(z) = g_1(z^p)$ for some polynomials $f_1, g_1$. Hence $\phi(z) = (f_1 / g_1)(z^p)$ as desired. So let us suppose that $D(g)$ is not identically zero. Then~\eqref{Eq: Anti-Leibniz} implies $f / g = D(f) / D(g)$, which contradicts the fact that $D(f)$ and $D(g)$ have strictly smaller degree than $f$ and $g$, respectively.
\end{proof}

\begin{lem}
\label{Lem: Reciprocal}
    Let $\phi \in F(z)$ be a rational function. Then the set of finite critical points of $\phi$ and $1 / \phi$ agree. In symbols, $\Crit{\phi} = \Crit{1 / \phi}$. 
\end{lem}

\begin{proof}
    Let $x \neq \infty$, and choose a fractional linear transformation $\sigma_1$ with coefficients in $F$ such that $\sigma_1(\phi(x)) \neq 0, \infty$. Choose polynomials $f, g \in F[z]$ with no common root such that $\sigma_1 \circ \phi = f / g$. If we set $\sigma_2(z) = \sigma_1(1/z)$, then $\sigma_2 \circ (1 / \phi) = \sigma_1 \circ \phi$. It follows that $x$ is a finite critical point for $\phi$ if and only if it is a finite critical point for $1 / \phi$. 
\end{proof}

    The following example illustrates the theorem for polynomial functions, and it will form the base case for the inductive proof of the theorem.

\begin{example}[Polynomials]
\label{Ex: Polynomials}
    Let $\phi \in F[z]$ be a polynomial with no finite critical point. Then $\phi'(z)$ cannot have any zeros, which is to say that $\phi' = a \in F \smallsetminus \{0\}$. The kernel of the derivative map on $F[z]$ is precisely the $F$-subalgebra $F[z^p]$ . 
Hence $\phi(z) = q_0(z^p) + az$ for some polynomial $q_0$. 
\end{example}

\begin{proof}[Proof of Theorem~\ref{Thm: No Finite Critical}]
	First we show that any rational function of the form~\eqref{Eq: Special Form} has no finite critical point. If $n = 0$, then $\phi(z) = q_0(z^p) + az$. So $\phi' = a \neq 0$ and the result follows. Now fix $\ell \geq 0$ and suppose that a rational function of the form \eqref{Eq: Special Form} with $n = \ell$ has no finite critical point. We will deduce the desired statement if $\phi(z) = [q_0(z^p), q_1(z^p), \ldots, q_{\ell + 1}(z^p)+ az]$. Indeed, observe that the rational function
	\benn
		 \frac{1}{\phi(z) - q_0(z^p)} = [q_1(z^p), q_2(z^p), \ldots, q_{\ell+1}(z^p) + az]
	\eenn
has no finite critical point by the induction hypothesis. The critical points of $\phi$ are related by 
    \benn
        \ba
            \Crit{\phi} &= \Crit{\phi(z) - q_0(z^p)} \quad \text{(Lemma~\ref{Lem: Kernel})}, \\
                &= \Crit{\frac{1}{\phi(z) - q_0(z^p)}} \quad \text{(Lemma~\ref{Lem: Reciprocal})},
        \ea
    \eenn
and hence $\phi$ has no finite critical point. 

    The converse is a consequence of the following fact. 
 
\begin{obs*}
	Let $\psi = a / b \in F(z)$ be a rational function, where $a, b$ are polynomials with no common root and $\deg(b) \geq 1$. If $\psi$ has no finite critical point, then there exist polynomials $q_0, r \in F[z]$ with $0 \leq \deg(r) < \deg(b)$ such that $\psi(z) = q_0(z^p) + r(z) / b(z)$.\footnote{As with any good student of valuation theory, we subscribe to the convention that the degree of the zero polynomial is $- \infty$.}  
\end{obs*}

\begin{proof}[Proof of the Key Fact]
	The division algorithm gives polynomials $q, r \in F[z]$ such that $a = bq + r$ and $\deg(r) < \deg(b)$. Write $q(z) = q_0(z^p) + q_1(z)$, where $p$ does not divide the exponent of any monomial in $q_1$ (Example~\ref{Ex: Polynomials}). Then
	\[
		\psi' = q_1' + \frac{r'b - rb'}{b^2} = \frac{b^2q_1' + r'b - rb'}{b^2}. 
	\]
Now $b$ is assumed to be nonconstant, and it cannot have a repeated root --- else $\psi$ would have this root as a finite critical point. Let us assume for the sake of a contradiction that $q_1 \neq 0$. Then
	\[
		\deg(r'b - rb') \leq \deg(b) + \deg(r) - 1 \leq 2\deg(b) - 2 < \deg(b^2q_1').
	\]
Since $\gcd(b, b^2q_1' + r'b - rb') = \gcd(b, rb') = 1$, the degree of the numerator of $\psi'$ agrees with that of $b^2 q_1'$. In particular, the numerator of $\psi'$ is nonconstant, and hence $\psi$ has a finite critical point. This contradiction shows $q_1 = 0$. If $r = 0$ too, then $a / b = q_0$ is an inseparable polynomial, which has infinitely many (finite) critical points. This contradiction completes the proof of the Key Fact. 
\end{proof}
 
We now return to the proof of the theorem, which will proceed by induction on the degree of the denominator of our rational function. Suppose that $\phi  \in F(z)$ is a rational function with no finite critical point whose denominator has degree~0; i.e., $\phi$ is a polynomial. Then Example~\ref{Ex: Polynomials} shows that $\phi(z) = q_0(z^p) + az = [q_0(z^p) + az]$ with $a \neq 0$, and we are finished. 

Next assume that the result holds for every rational function whose denominator has degree at most $\ell \geq 0$. Let $\phi = f / g$ be a rational function with no finite critical point. We assume that $f$ and $g$ have no common root, and that $\deg(g) = \ell+1$. By the Key Fact there exist polynomials $q_0, r$ such that $\phi(z) = q_0(z^p) + r(z) / g(z)$, where $0 \leq \deg(r) < \deg(g)$. Lemmas~\ref{Lem: Kernel} and~\ref{Lem: Reciprocal} show that
	\benn
		\Crit{\phi} = \Crit{\frac{r}{g}} = \Crit{\frac{g}{r}}.
	\eenn
Hence $g / r$ has no finite critical point. As the degree of $r$ is strictly less than $\ell+1$, we may apply the induction hypothesis to deduce that 
	\[
		\frac{g(z)}{r(z)}  = [q_1(z^p), q_2(z^p), \ldots, q_n(z^p) + az]
	\]
for some polynomials $q_1, \ldots, q_n$ and a nonzero constant $a$. But then
	\benn
		\ba
			\phi =  q_0(z^p) + \frac{1}{g(z) / r(z)} 
			&= q_0(z^p) + \frac{1}{ [q_1(z^p), q_2(z^p), \ldots, q_n(z^p) + az]} \\
			&= 
			[q_0(z^p), q_1(z^p), \ldots, q_n(z^p)+az],
		\ea
	\eenn
which completes the induction step. 
\end{proof}


\section{The Geometry of the Unicritical Locus}
\label{Sec: Unicritical Geometry}

	All varieties in this section will be defined over~$\FF_p$ for a fixed prime~$p$. 
	
\subsection{Definition of $U_d$}
\label{Sec: U_d definition}
	
	A rational function of degree $d \geq 1$ can be written as
		\begin{equation}
		\label{Eq: Coordinates}
			\phi(z) = \frac{f(z)}{g(z)} = \frac{a_dz^d + \cdots + a_0}{b_dz^d + \cdots + b_0},
		\end{equation}
where $f$ and $g$ are uniquely determined up to a common scalar multiple and $\Res(f, g) \neq 0$. Here $\Res(f,g)$ denotes the degree-$(d,d)$ resultant of $f$ and $g$; see, e.g., \cite[App.~A]{Milnor_Rational_Maps_1993} or \cite[\S2.4]{Silverman_Dynamics_Book_2007}. Following Milnor, we write $\Rat_d$ for the space of rational functions of degree~$d$. It is the Zariski open subvariety of $\PP^{2d+1}$ defined by 
	\begin{equation}
	\label{Eq: Rat_d}
		\Rat_d = \PP^{2d+1} \smallsetminus \{\Res(f, g) = 0\},
	\end{equation}
where we take $(a_d: \cdots : a_0 : b_d : \cdots : b_0)$ to be homogeneous coordinates on $\PP^{2d+1}$. We will abuse notation by writing $\phi$ for the corresponding point in $\Rat_d$. 
	
	This description of $\Rat_d$ works over an arbitrary field (or indeed over $\Spec \ZZ$ --- see \cite{Silverman_Rational_Maps_1998}), but we now look at two subloci that are special to fields of characteristic~$p$. First define $\Insep_d$ to be the subvariety of $\Rat_d$ corresponding to inseparable rational functions. Evidently it is given by an intersection of appropriate coordinate hyperplanes:
	\[
		\Insep_d = \{\phi \in \Rat_d : a_i = b_i = 0 \text{ whenever $p \nmid i$}\}.
	\]
Note that $\Insep_d$ is the empty variety if $p$ does not divide $d$.

	Next we want to define the locus of unicritical functions, denoted $U_d \subset \Rat_d$. As a set, it is clear what this should mean, but we must endow it with the structure of an algebraic variety. Intuitively, we fiber $\Rat_d$ according to the set of critical points of a rational function, and then define $U_d$ to be those fibers corresponding to a single critical point. We now execute this strategy, although it turns out to be more natural to fiber $\Rat_d$ using the polynomial defining the critical points instead. 
	
	The finite critical points of $\phi = f / g \in \Rat_d$ are the roots of the polynomial $f'(z)g(z) - f(z)g'(z) = \sum_{i = 0}^{2d-2} c_i z^i$. Each $c_i$ is a homogeneous quadratic polynomial in the coefficients $a_i, b_j$. This allows us to define a morphism $\omega: \Rat_d \smallsetminus \Insep_d \to \PP^{2d-2}$ by 
	\[
		\omega(\phi) = (c_{2d-2} : \cdots : c_0).
	\] 
Note that $\omega$ is well defined because $f'g - fg'$ vanishes precisely on the inseparable locus (Lemma~\ref{Lem: Kernel}). 

	Next define a morphism $\theta : \PP^1 \to \PP^{2d-2}$ by 
	\[
		\theta(s : t) = \left( \theta_{2d-2}(s,t) : \cdots : \theta_0(s,t) \right), \qquad \text{where }
		(tz - s)^{2d-2} = \sum_{i = 0}^{2d-2} \theta_i(s,t) z^i.
	\]
Let $C$ be the image of $\theta$ in $\PP^{2d-2}$; it is a rational curve when $d > 1$.  By construction, if $\omega(\phi) = \theta(s : t) \in C$, then $(s : t)$ is the unique critical point of $\phi$. 

	Now define the unicritical locus $U_d$ by the fiber product square
	\begin{equation}	
	\label{Diagram: Fiber Square}
		\xymatrix{
			U_d \ar[r] \ar[d] & \Rat_d \smallsetminus \Insep_d \ar[d]^{\omega} \\
			C \ar[r] & \PP^{2d-2}.
		}
	\end{equation}		
Since the bottom arrow is a closed immersion, we find $U_d$ is a closed subvariety of $\Rat_d \smallsetminus \Insep_d$. 

\begin{remark}
	The map $\theta: \PP^1 \to \PP^{2d-2}$ is not a closed immersion in general. But one can show that it factors as $\tilde \theta \circ \mathrm{Frob}_{\PP^1 / \FF_p}^r$ for some $r \geq 0$, where $\mathrm{Frob}_{\PP^1 / \FF_p}: \PP^1_{\FF_p} \to \PP^1_{\FF_p}$ is the relative Frobenius map and $\tilde \theta$ is a closed immersion. In particular, $C \cong \PP^1$ when $d > 1$.   
\end{remark}

\begin{remark}
	Observe that this construction gives $U_1 = \Rat_1 = \PGL_2$. 
\end{remark}

\begin{remark}
	Over $\CC$, the morphism $\Rat_d \stackrel{\omega}{\longrightarrow} \PP^{2d-2}$ has dense image. As $\phi$ and $\sigma \circ \phi$ have the same critical points for any $\sigma \in \mathrm{SL}_2(\CC)$, the morphism $\omega$ descends to the quotient $\mathrm{SL}_2 \setminus \Rat_d \to \PP^{2d-2}$. The latter morphism is quasi-finite of generic degree $\frac{1}{d}\binom{2d-2}{d-1}$, the $d$th Catalan number \cite{Goldberg_Catalan_1991}. In particular, the fibers of $\omega$ are $3$-dimensional.
	
	In positive characteristic, the fibers of $\omega$ can be far more wild. For example, Theorem~\ref{Thm: Unicritical Geometry} implies that the fiber of $\omega$ over any point of $C$ has dimension at least~4, and that this dimension grows with~$d$. See also \cite{Osserman_Deformations_2005, Osserman_Ramification_char_p_2006} for  discussions of postcomposition classes of rational functions in characteristic~$p$ with more general ramification structure. Note that, with the exception of the simplest case $p = d = 2$, unicritical rational functions are not treated in these works. 
\end{remark}


\subsection{Preliminaries for the Main Results}
\label{Sec: Preliminaries}
	
	For each field $F$ of characteristic~$p$ and each $c \in \PP^1(F)$, write $U_d^{(c)}$ for the fiber of $U_d \to C$ over the point $\theta(c)$.  Of greatest interest will be the case $c = \infty$, so that $\theta(\infty) = (0 : \cdots : 0 : 1)$. The $F$-rational points of $U_d^{(c)}$ are unicritical rational functions of degree~$d$ with critical point~$c$.
	
\begin{lem}
\label{Lem: Smaller space}
	Let $F$ be a field of characteristic~$p$. 
	If $\phi \in F(z)$ is a rational function with no finite critical point, then there exist polynomials
	$f_1, f_2, g_1, g_2 \in F[z]$ such that 
		\begin{equation}
		\label{Eq: Linear subspace}
			\phi(z) = \frac{f_1(z^p) + zf_2(z^p)}{g_1(z^p) + zg_2(z^p)}, 
		\end{equation}
	and $f_2g_1 - f_1g_2$ is a nonzero constant function. 
\end{lem}

\begin{proof}
	If $\phi$ is a polynomial, then Theorem~\ref{Thm: No Finite Critical} shows we may write $\phi$ in this form with $f_2$ a nonzero constant, $g_1 = 1$, and $g_2 = 0$. More generally, we note that if $\phi$ can be written in the desired form, then so can $q(z^p) + 1 / \phi(z)$ for any polynomial $q \in F[z]$. Theorem~\ref{Thm: No Finite Critical} shows that any $\phi \in \uni(F)$ can be constructed recursively from a polynomial by inversion and addition of an inseparable polynomial, so the proof of the first statement is complete after an appropriate induction. 
	
	If $\phi$ has no finite critical point, then the formal derivative $\phi'$ has no finite zero. In particular, the 
	numerator of $\phi'$ must be a nonzero constant. The final statement is now a direct calculation using the 
	quotient rule.
\end{proof}

\begin{lem}
\label{Lem: Dimension bound}
	Fix a positive integer $d \equiv 0 \text{ or } 1 \pmod p$. 
	Then the dimension of each irreducible component of $\uni$ is at least 
	\[
			\begin{cases} 
				2 + 2d / p & \text{if } p \mid d \\
				3 + 2(d-1) / p & \text{if } p \mid d-1.
			\end{cases}
	\]
\end{lem}

\begin{proof}
	We prove the case $p \mid d$, and leave the other case to the reader.
	
	Consider the linear subspace $L_d \subset \Rat_d$ defined by
		\[
			L_d = \{ \phi \in \Rat_d : a_i = b_i = 0 \text{ if $i \not\equiv 0$ or $1\pmod{p}$}\},
		\]
where the coordinates $a_i, b_j$ are as in~\eqref{Eq: Coordinates}. Then $L_d$ consists of the rational functions of the form~\eqref{Eq: Linear subspace}. For a generic element $\phi \in L_d$, we must have $\deg(f_1) =  \deg(g_1) = d / p$, and $\deg(f_2) = \deg(g_2) = d / p - 1$. 

	The previous lemma shows that $\phi \in \uni$ is equivalent to $f_2g_1 - f_1g_2$ being a nonzero constant. Since  $\deg(f_2g_1 - f_1g_2) = 2d / p - 1$, the vanishing of the nonconstant coefficients of $f_2g_1 - f_1g_2$ is a codimension $2d / p - 1$ condition on $L_d$. That is, $\uni$ is cut out from a Zariski open subset of $L_d$ by at most $2d / p - 1$ equations. Therefore, each irreducible component of $\uni$ has dimension at least 
		$\dim(L_d) - (2d/p - 1) =  2 + 2d/p$.
\end{proof}
	
	Given an arbitrary rational function $\phi$ of degree~$d$, let $\phi = [f_0, \ldots, f_n]$ be its continued fraction expansion. The \textbf{signature} of $\phi$ is the tuple of nonnegative integers given by 
	\[
		\sign(\phi) = \left(\deg^+(f_0), \deg(f_1), \ldots, \deg(f_n) \right). 
	\]
Here we have written $\deg^+(f) = \max \{\deg(f), 0\}$ for simplicity. Observe that if $\phi$ is a rational function with signature $\kappa = (\kappa_0, \ldots, \kappa_n)$, then $\deg(\phi) = \sum \kappa_i$. In particular, there are only finitely many signatures for rational functions of fixed degree. 
	
	Now let $F$ be a field of characteristic~$p$, and suppose $\phi \in \uni(F)$. By Theorem~\ref{Thm: No Finite Critical}, we may write $\phi(z) = [q_0(z^p), q_1(z^p), \ldots, q_n(z^p)+ az]$ for some polynomials $q_i \in F[z]$ and a nonzero element $a \in F$. It follows that the signature of $\phi$ is 
	\[
		\sign(\phi) = \begin{cases}
			\left(p\cdot \deg^+(q_0), \ldots, p\cdot \deg(q_n) \right) & \text{if } p \mid d \\
			\left(p\cdot \deg^+(q_0),  \ldots, p\cdot \deg(q_{n-1}), 1) \right) & \text{if } p \mid d-1.
		\end{cases}
	\]
	
	Fix an integer $d \geq 1$ and a tuple $\kappa = (\kappa_1, \ldots, \kappa_n)$ of nonnegative integers. Define $\uni(\kappa)$ to be the space of rational functions of degree~$d$ with signature~$\kappa$ and no finite critical point. 
	
\begin{prop}
\label{Prop: Zariski open}
	Fix a positive integer $d \equiv 0 \text{ or }1 \pmod{p}$. If $p \mid d$, define $\kappa^\circ = (0, \underbrace{p, \ldots, p}_{d / p})$. If $p \mid d-1$, define $\kappa^\circ = (0, \underbrace{p, \ldots, p}_{(d-1)/p}, 1)$. Then $\uni(\kappa^\circ)$ is nonempty and Zariski open in $\uni$. 
\end{prop}

\begin{proof}
	The proofs of the two statements are similar, so we only treat the case $p \mid d$. Evidently $\uni(\kappa^\circ)$ is nonempty; for example, it contains the function $\phi(z) = [1, z^p, z^p, \ldots, z^p + z]$. 
	
	In general, write $\phi$ as in \eqref{Eq: Coordinates}, and also suppose $\phi = [q_0(z^p), q_1(z^p), \ldots, q_n(z^p) + az]$. We proceed by induction to show that $\deg(q_0) = 0$ and $\deg(q_i) = 1$ for $i > 0$ are open conditions. The requirement that $q_0$ be a constant is equivalent to $b_d \neq 0$. Now suppose the conditions on $q_0, \ldots, q_\ell$ are all open for some $\ell \geq 0$. The construction of the continued fraction expansion of $\phi$ gives 
	\[
		\phi =  q_0 + \frac{1}{q_1(z^p) +  \frac{1}{\ddots q_\ell(z^p) + \frac{1}{r_{\ell - 1}(z) / r_\ell(z)} } },
	\]
where $r_{\ell - 1}, r_\ell$ are polynomials in $z$ with $\deg_z(r_{\ell - 1}) > \deg_z(r_\ell)$ and whose coefficients are rational functions in the $a_i$'s and $b_j$'s. In fact, since $\phi \in \uni$, we see that $\deg_z(r_\ell) \leq \deg_z(r_{\ell - 1}) - p$. The coefficients of $r_{\ell - 1}, r_\ell$ are regular functions on a Zariski open subset of $\Rat_d$, and equality holds in this last inequality if and only if some rational function in the $a_i$'s and $b_j$'s does not vanish --- namely, the leading coefficient of $r_\ell$. Hence, on a Zariski open subset of $\Rat_d$, we find that $r_{\ell - 1} / r_\ell = q_{\ell + 1}(z^p) + r_{\ell + 1} / r_\ell$ for some linear polynomial $q_{\ell +1}$ and some polynomial $r_{\ell + 1}$ with $\deg_z(r_{\ell + 1}) < \deg_z(r_\ell)$. By induction, one obtains the desired result. 
\end{proof}


\subsection{Proofs of the Main Results}
\label{Sec: Proofs}

\begin{proof}[Proof of Theorem~\ref{Thm: Unicritical Geometry}]
	We prove the theorem when $p \mid d$. The other case is similar.
	
	Given a rational function $\phi \in U_d(F)$ with critical point $c$, we can precompose with an element $\sigma \in \PGL_2(F)$ to see that $\phi \circ \sigma$ has critical point $\sigma^{-1}(c)$.  As this action is algebraic and invertible, the fibers of $U_d$ over $\theta(c)$ and over $\theta(\sigma^{-1}(c))$ are isomorphic. In particular, choosing $\sigma$ so that $\sigma(\infty) = c$, we find that $U_d^{(c)} \cong \uni$. Since $U_d$ is fibered over the rational curve $C$ and all of the fibers are isomorphic, it suffices to show that $\uni$ is irreducible and rational of dimension $2 + 2d / p$.
	
	Fix a signature $\kappa = (\kappa_0, \ldots, \kappa_n)$ such that $\uni(\kappa)$ is nonempty. Let $\Poly_i$ be the space of polynomials in one variable of degree~$i$; evidently the coefficients of the polynomial identify $\Poly_i$ with $(\Aff^1 \smallsetminus \{0\}) \times (\Aff^1)^i$. Define a morphism 
	\begin{align*}
		\Poly_{\kappa_0 / p} \times \cdots \times \Poly_{\kappa_n / p} \times \left(\Aff^1 \smallsetminus\{0\}\right)
			&\to \uni(\kappa) \\
		 (q_0, \ldots, q_n, a) &\mapsto [q_0(z^p), \ldots, q_n(z^p) + az].
	\end{align*}
Theorem~\ref{Thm: No Finite Critical} shows this map is bijective on geometric points, and hence $\uni(\kappa)$ is the birational image of a variety of dimension $1 + \sum (\kappa_i / p+ 1) = 2 + d / p + n$.

	Proposition~\ref{Prop: Zariski open} shows that $\uni(\kappa^{\circ})$ is Zariski open in $\uni$, and the last paragraph shows it is irreducible of dimension 
$2 + 2d / p$. If $\kappa = (\kappa_0, \ldots, \kappa_n) \neq \kappa^{\circ}$, then $n < d / p$. The previous paragraph shows
	\[
		\dim(\uni(\kappa)) = 2 + d/p + n < 2 + 2d / p.
	\]
	
	Write $\uni$ as the union of two Zariski closed subsets: 
	\[
		\uni = \overline{\uni(\kappa^\circ)} \cup \left[ \uni \smallsetminus \uni(\kappa^{\circ}) \right].
	\]
The second set is the image of finitely many varieties of dimension strictly smaller than $2 + 2d / p$, so that each of its irreducible components has dimension strictly smaller than $2 + 2d / p$. Since each of the irreducible components of $\uni$ must have dimension at least $2 + 2d / p$ (Lemma~\ref{Lem: Dimension bound}), we conclude that $\uni(\kappa^{\circ})$ is dense in $\uni$. That is, $\uni$ is irreducible and rational of dimension $2 + 2d/p$.
\end{proof}

\begin{proof}[Proof of Theorem~\ref{Thm: Unicritical Conjugacy}]
	Write $\pi: U_d \to \mathfrak{U}_d = U_d / \SL_2$ for the quotient map. As $U_d \subset \Rat_d$ is an invariant subset of the stable locus for the conjugation action of $\SL_2$ \cite{Silverman_Rational_Maps_1998}, it follows that $\pi: U_d \to \mathfrak{U}_d$ is a geometric quotient. In particular, it is surjective, so that $\mathfrak{U}_d$ is irreducible by Theorem~\ref{Thm: Unicritical Geometry}. 
	
	Let $Y \subset U_d^{(\infty)}(\kappa^\circ)$ be the subvariety of rational functions $\phi$ satisfying $\phi(\infty) = 0$ and $\phi(0) = 1$. The theorem is immediately implied by the following three claims about $Y$.
	
	\textbf{Claim 1: $Y$ is a rational variety.} We treat the case $p \mid d$; the other case is similar. 
Consider the morphism 
	\begin{align*}
		\left(\Aff^1 \smallsetminus\{0\}\right) \times \Poly_{1} \times \cdots \times \Poly_{1} \times \left(\Aff^1 \smallsetminus\{0\}\right)
			&\to Y \\
		 (\alpha, q_2, \ldots, q_n, \gamma) 
		 &\mapsto [0, \alpha z^p + \beta(q_2, \ldots, q_n), q_2(z^p), \ldots, q_n(z^p) + \gamma z],
	\end{align*}
where $\beta: \Poly_1 \times \cdots \times \Poly_1 \to \Aff^1$ is given by
	\[
		\beta(q_2, \ldots, q_n) = 1 - \frac{1}{q_2(0) + \frac{1}{\ddots + \frac{1}{q_n(0)}}}.
	\]
Evidently this morphism is injective on geometric points, and we claim it is also a surjection. To see it, observe that any $\phi \in Y(F)$ has a continued fraction expansion of the form 
	\[
		\phi(z) = [q_0(z^p), q_1(z^p), q_2(z^p), \ldots, q_n(z^p) + \gamma z]
	\]
for some constant polynomial $q_0$, linear polynomials $q_1, \ldots, q_n$, and nonzero $\gamma \in F$. To be in $Y$, $\phi$ must satisfy $\phi(\infty) = 0$ and $\phi(0) = 1$. The former condition translates into $q_0 = 0$, while the latter is then equivalent to
	\[
		1 = \phi(0) = \frac{1}{q_1(0) + \frac{1}{q_2(0) + \frac{1}{\ddots + \frac{1}{q_n(0)}}}}.
	\]
Solving for $q_1(0)$ gives $q_1(0) = \beta(q_2, \ldots, q_n)$. It follows that we may write $q_1(z^p) = \alpha z^p + \beta(q_2, \ldots, q_n)$ for some nonzero $\alpha \in F$. This completes the proof. 
	
	\textbf{Claim 2: $\pi|_Y$ is dominant.} Let $c, \phi(c), \phi^2(c)$ be the critical point of $\phi$, its critical value, and the image of the critical value under $\phi$, respectively. For a generic $\phi \in U_d$, these three points are distinct. Let $\sigma \in \SL_2$ be an automorphism that maps $\infty, 0, 1$ to $c, \phi(c), \phi^2(c)$, respectively. Then $\psi = \sigma^{-1} \circ \phi \circ \sigma \in U_d^{(\infty)}$, $\psi(\infty) = 0$, and $\psi(0) = 1$. As $U_d^{(\infty)}(\kappa^\circ)$ is open in $U_d^{(\infty)}$ (Proposition~\ref{Prop: Zariski open}), we see that a generic $\phi \in U_d$ is conjugate to an element of $Y$. Hence $\pi|_Y$ is dominant. 
	
	\textbf{Claim 3: $\pi|_Y$ is injective on geometric points.} Let $F$ be an algebraically closed field, and let $\phi, \psi \in Y(F)$ be such that $\pi(\phi) = \pi(\psi)$. Then there exists $\sigma \in \SL_2(F)$ such that $\sigma \circ \psi = \phi \circ \sigma$. Since $\infty$ is the unique critical point of both $\phi$ and $\psi$, since $\phi(\infty) = \psi(\infty) = 0$, and since $\phi(0) = \psi(0) = 1$, it follows that $\sigma$ fixes the three points $0, 1, \infty$. Hence $\sigma$ is the identity, so that $\phi = \psi$. 
\end{proof}

\begin{proof}[Proof of Theorem~\ref{Thm: Classes}]
	We treat only the case $p \mid d$. Without loss of generality, we may assume $c = \infty$. A generic element of $U_{d,e}^{(\infty)}$ does not fix $\infty$. Let $\phi(z) = [q_0(z^p), q_1(z^p), \ldots, q_n(z^p) + az]$ be the continued fraction expansion of~$\phi$, according to Theorem~\ref{Thm: No Finite Critical}. Then $q_0$ is constant and $e = p \cdot \deg(q_1)$. Moreover, a generic element of $U_{d,e}^{(\infty)}$ has signature $(0, e, \underbrace{p, \ldots, p}_{(d-e) / p})$. (Compare with Proposition~\ref{Prop: Zariski open}.) Counting the coefficients of the polynomials that define a unicritical rational function with signature $(0, e, p, \ldots, p)$ as in the proof of Theorem~\ref{Thm: Unicritical Geometry}, we find that
	\[
		\dim(U_{d,e}^{(\infty)}) = 1 + \left(\frac{e}{p}+1\right) + 2\frac{d- e}{p} + 1 = 3 + \frac{2d - e}{p}.
	\]
Since $\SL_2$ has dimension~3 and acts on $U_{d,e}^{(\infty)}$ without fixed points, the result follows. 
\end{proof}


\subsection{Example: Unicritical Functions of Degree~$p$}
\label{Sec: Example}

	Using the ideas from the previous sections, we may realize $U_p^{(\infty)}$ as a Zariski open subset of a quadric hypersurface in $\PP^5$. Indeed, if $\phi \in U_p^{(\infty)}$, then Lemma~\ref{Lem: Smaller space} shows that 
	\[
		\phi(z) = \frac{a_pz^p + a_1z + a_0}{b_pz^p + b_1z + b_0}
	\]
for some coefficients $a_i, b_j$. Such a rational function may be identified with the point $(a_p : a_1 : a_0 : b_p : b_1 : b_0) \in L_p \subset \PP^5$. Write $\Res_p(\phi)$ for the resultant of the rational function $\phi$. The numerator of $\phi'$ is $(a_1 b_p - a_p b_1)z^p + (a_1b_0 - a_0 b_1)$. Then
	\[
		U_p^{(\infty)} =  \{a_1 b_p = a_p b_1\} \cap \{ \Res_p \neq 0\} \cap \{ a_1b_0 \neq a_0 b_1\}  \subset L_p.
	\]	
In particular, $U_p^{(\infty)}$ is irreducible of dimension~4. Moreover, the critical point must have ramification index~$p$, so that in the notation of Theorem~\ref{Thm: Classes}, we have $U_{p, p}^{(\infty)} = U_p^{(\infty)}$. It follows that the space of unicritical postcomposition classes $\SL_2 \backslash U_{p, p}^{(\infty)}$ has dimension~1, as expected.  

	Finally, let $\mathrm{Aff}_2$ be the subgroup of $\SL_2$ that stabilizes $\infty$; more concretely, it consists of maps of the form $z \mapsto Az + B$. Then the space  of unicritical conjugacy classes  satisfies $\mathfrak{U}_p = U_p / \SL_2 = U_p^{(\infty)} / \mathrm{Aff}_2$, which has dimension~2.  


\section{An Application}
\label{Sec: Application}

	Here we present the proof of Corollary~\ref{Cor: Application} after some necessary definitions.

	Let $\varphi \in k(z)$ be a nonconstant rational function written as $\varphi = f / g$ with $f, g \in k[z]$ having no common root. After multiplying $f$ and $g$ by a suitable element of $k$,  we may suppose further that $f$ and $g$ have coefficients in $\OO_k$ and that $f$ or $g$ has a coefficient in $\OO_k^\times = \{x \in \OO_k : |x| = 1\}$. This normalization of $f$ and $g$ is unique up to simultaneous multiplication by an element of $\OO_k^\times$. Write $\tilde f$ and $\tilde g$ for the images of $f$ and $g$ in $\tilde k[z] = \OO_k[z] / \mm[z]$. The \textbf{reduction} of $\varphi = f/g$ is defined to be the map $\widetilde{\varphi}: \PP^1(\tilde k) \to \PP^1(\tilde k)$ given by 
	\[
		\widetilde{\varphi} = \begin{cases}
				\tilde f / \tilde g & \text{if } \tilde g \neq 0 \\
				\infty & \text{if } \tilde g = 0.
			\end{cases}
	\]
We say that $\varphi$ has \textbf{good reduction} if $\deg(\varphi) = \deg(\widetilde{\varphi})$. 

\begin{proof}[Proof of Corollary~\ref{Cor: Application}]
	Let $\tilde c \in \PP^1(\tilde k)$ be the image of the critical points of $\varphi$ under the reduction map. Without loss of generality, we may replace $\varphi$ with $\sigma^{-1} \circ \varphi \circ \sigma$ for some $\sigma \in \PGL_2(\OO_k)$ in order to assume that $\tilde c = 0$. More precisely, if $\tilde c \neq \infty$, then we take $\sigma(z) = z + c$ for some critical point $c$ of $\varphi$; if $\tilde c = \infty$, then we take $\sigma(z) = 1 / z$. Conjugating by $\sigma \in \PGL_2(\OO_k)$ preserves good reduction (as well as the the degree of $\varphi$). 
	
	Suppose first that the reduction $\widetilde{\varphi}$ is inseparable; that is, there exists a rational function $\psi \in \tilde k(z)$ such that $\widetilde{\varphi}(z) = \psi(z^p)$. Then
	\[
		\deg(\varphi) = \deg(\widetilde{\varphi}) = p \cdot \deg(\psi), 
	\]
so that $p \mid \deg(\varphi)$. 

	Now suppose that the reduction $\widetilde{\varphi}$ is separable. Then $\widetilde{\varphi}$ has only finitely many critical points by the Hurwitz formula. Writing $\varphi = f / g$ with $f, g \in \OO_k$ normalized, we see that the critical points of $\varphi$ are given by the roots of the polynomial
	\[
		f'(z) g(z) - f(z) g'(z) = a \prod (z - c_i),
	\]
where $a \in \OO_k$ and $c_1, \ldots, c_{2d - 2} \in \OO_k$ are the critical points of $\varphi$. Here we have written $d = \deg(\varphi) = \deg(\widetilde{\varphi})$. Since $\widetilde{\varphi} = \tilde f / \tilde g$ has the same degree as $\varphi$, it follows that the critical points of $\widetilde{\varphi}$ are given by the roots of the polynomial 
	\[
		\tilde f'(z) \tilde g(z) - \tilde f(z) \tilde g'(z) = \tilde a \prod (z - \tilde c_i) = \tilde a z^{2d-2}.
	\]
Note that $\tilde a \neq 0$ since $\widetilde{\varphi}$ is separable. It follows that $\widetilde{\varphi}$ has a unique critical point. An application of Corollary~\ref{Cor: Congruence} completes the proof. 	
\end{proof}

\bibliographystyle{plain}
\bibliography{xander_bib}

\end{document}